\documentclass[11pt]{article}
\usepackage{amssymb}
\usepackage{amsfonts}
\usepackage{amsmath}
\usepackage{amsmath}
\usepackage{verbatim}
\usepackage{color}
\usepackage[active]{srcltx}

\usepackage{ mathrsfs }

\include{srctex}
\allowdisplaybreaks[4]
\newtheorem{theorem}{Theorem}[section]

\newtheorem{corollary}[theorem]{Corollary}

\newtheorem{definition}[theorem]{Definition}

\newtheorem{lemma}[theorem]{Lemma}

\newtheorem{problem}[theorem]{Problem}
\newtheorem{proposition}[theorem]{Proposition}
\newtheorem{remark}[theorem]{Remark}

\newtheorem{hypothesis}[theorem]{Hypothesis}

\let\Section=\section
\def\section{\setcounter{equation}{0}\Section}
\newenvironment{proof}[1][Proof]{\textbf{#1.} }{\ \rule{0.5em}{0.5em}}
\newcommand{\R}{\mathbb{R}}

\def\RR{\mathbb{R}}

\def\EE{\mathbb{E}}

\def \eref#1{\hbox{(\ref{#1})}}

\begin{document}

\title{On the intermittency front of stochastic heat equation driven by colored noises }

\author{Yaozhong {\sc Hu}\thanks{Y.  Hu is
partially supported by a grant from the Simons Foundation
\#209206.}, \  Jingyu {\sc Huang} \ and \  David
{\sc Nualart}\thanks{ D. Nualart is supported by the
NSF grant DMS1208625 and the ARO grant  FED0070445. \newline
  Keywords:   Stochastic heat equation, Feynman-Kac formula, Intermittency front,  Malliavin calculus,   comparison principle.   }  \\
Department of Mathematics \\
University of Kansas \\
Lawrence, Kansas, 66045 USA}
\date{}
\maketitle

\begin{abstract}
We study the propagation of  high peaks (intermittency front) of the solution to a stochastic heat equation driven by multiplicative centered Gaussian noise in $\RR^d$. The noise is assumed to have a general homogeneous covariance in both time and space, and the solution is interpreted in the senses of the Wick product.  We give some estimates for the upper and lower bounds of the propagation speed, based on a moment formula of the solution. When the space covariance is given by a Riesz kernel, we give  more precise bounds for the propagation speed. 
\end{abstract}

\section{Introduction}
We consider the stochastic heat equation in $\RR^d$ driven by a general multiplicative centered Gaussian noise (parabolic Anderson model)
\begin{equation}\label{SHE}
\frac{\partial u}{\partial t} = \frac{1}{2}\Delta u + \lambda u   \diamond \dot {W}\,,
\end{equation}
with a continuous and nonnegative initial condition $u_0$ of compact support. The covariance of the noise $\dot W$ can be informally written as 
\begin{equation*}
\EE \left [ \dot{W}_{t,x} \dot{W}_{s,y}\right] = \gamma(s-t) \Lambda(x-y)\,,
\end{equation*}
and the product appearing in \eref{SHE} is interpreted in the  Wick sense.

In this paper we are interested in the position of the high peaks that are farthest away from the origin. The propagation of the farthest high peaks was  first considered by Conus and Khoshnevisan in \cite{CK} for a one dimensional heat equation driven by space-time white noise, where it is shown that  there are intermittency fronts that move linearly with time as $\alpha t$. Namely, for any fixed $p \in [2, \infty)$, if $\alpha$ is sufficiently small, then the quantity $\sup_{|x|> \alpha t}\EE (|u(t,x)|^p)$ grows exponentially fast as $t $ tends to $\infty$; whereas the preceding quantity vanishes exponentially fast  if $\alpha$ is sufficiently large. To be more precise, the authors of   \cite{CK}  define for every $\alpha > 0$, 
\begin{equation}\label{linear growth index}
\mathscr{S}(\alpha):= \limsup_{t \to \infty} \frac{1}{t} \sup_{|x|> \alpha t} \log \EE (|u(t,x)|^p)\,,
\end{equation}
and think of $\alpha_L$ as an intermittency lower front if $\mathscr{S}(\alpha) < 0$ for all $\alpha > \alpha_L$, and of $\alpha_U$ as an intermittency upper front if $\mathscr{S}(\alpha)>0$ whenever $\alpha < \alpha_U$. In \cite{CK} it is shown that for each real number $p \geq 2$, $ 0 < \alpha_U \leq \alpha_L < \infty$, and when $p=2$, some  bounds for $\alpha_L$ and $\alpha_U$ are given. In a later work by Chen and Dalang \cite{CD}, it is proved that when $p=2$, there exists a critical number $\alpha^* =\frac{\lambda^2}{2}$ such that $\mathscr{S} (\alpha) < 0$ when $\alpha> \alpha^*$ while $\mathscr{S}(\alpha) > 0$ when $\alpha < \alpha^*$ (this property was first conjectured in \cite{CK}). 
See also \cite{Kho} for a discussion of these facts.    

This paper is inspired by the aforementioned works. We are  interested in  the  multidimensional stochastic heat equation driven  by a colored noise, both in space and time, when the solution is interpreted in the Wick sense.  Our analysis will be based on the $p$th moment formula and Wiener chaos expansion of the solution,  obtained in \cite{HHNT}, as well as some small ball estimates.  Due to the presence of the time covariance, the propagation speed of the farthest high peaks may not be linear. Thus,  in contrast to \eref{linear growth index}, the inequality  $|x|> \alpha t$ there needs to be replaced by $|x|> \alpha t\theta_t$ for some suitable function $\theta_t$ (see Theorems \ref{thm:cor upper} and \ref{thm:cor lower} below for precise choice of $\theta_t$). When $\Lambda$ is the Riesz kernel, a better estimate of the intermittency lower front is obtained in Proposition \ref{prop: Riesz upper}. We would like to mention the work \cite{CD1}, where another nonlinear propagation speed (growth indices of exponential type) is studied. 

This paper is organized as follows: In Section 2, we set up some preliminaries for the structure of our Gaussian noises in equation \eref{SHE} and present some elements of  Malliavin calculus. We also prove the non-negativity of the solution to equation \eref{SHE}. Section 3 contains the main results of this paper, where we obtain some upper and lower bounds for the growth index. In the special case when the space covariance is a Riesz kernel, we give a more detailed computation for the upper bound of the growth index, and we see that the orders of $\lambda$ and $p$ in the estimate of the growth index are sharp.

\section{Preliminaries}

We first introduce some basic notions.   The Fourier
transform is defined with the normalization
\[ \mathcal{F}u ( \xi)  = \int_{\mathbb{R}^d} e^{- \imath \langle
   \xi, x \rangle} u ( x) d x, \]
so that the inverse Fourier transform is given by $\mathcal{F}^{- 1} u ( \xi)
= ( 2 \pi)^{- d} \mathcal{F}u ( - \xi)$. We denote by  $\mathcal{D}((0,\infty)\times \RR^{d})$ the space of infinitely differentiable functions with compact support on  $(0,\infty)\times \RR^{d}$.

\smallskip
On a complete probability space
$(\Omega,\mathcal{F},P)$ we consider a Gaussian noise $W$ encoded by a
centered Gaussian family $\{W(\varphi) , \, \varphi\in
\mathcal{D}((0,\infty)\times \RR^{d})\}$, whose covariance structure
is given by
\begin{equation}\label{cov}
\EE\left [ W(\varphi) \, W(\psi) \right]
= \int_{\RR_{+}^{2}\times\RR^{2d}}
\varphi(s,x)\psi(t,y)\gamma(s-t)\Lambda(x-y)dxdydsdt,
\end{equation}
where $\gamma: \RR \rightarrow \RR_+$ and $\Lambda: \RR^d \rightarrow
\RR_+$ are  non-negative definite
functions. We also assume that  the Fourier transform $\mathcal{F}\Lambda=\mu$ is a tempered
measure, that is, there is an integer   $m \geq 1$ such that
$\int_{\RR^d}(1+|\xi|^2)^{-m}\mu(d\xi)< \infty$. Our results also cover the case where $\gamma$  (or $\Lambda$ if $d=1$) is the Dirac  delta function, which corresponds to the time (or space) white noise.

\smallskip

 Let $\mathcal{H}$  be the completion of
$\mathcal{D}((0,\infty)\times\RR^d)$
endowed with the inner product
\begin{eqnarray}\label{innprod}
\langle \varphi , \psi \rangle_{\mathcal{H}}&=&
\int_{\RR_{+}^{2}\times\RR^{2d}}
\varphi(s,x)\psi(t,y)\gamma(s-t)\Lambda(x-y) \, dxdydsdt\\ \notag
&=&\frac{1}{(2\pi)^d}\int_{\RR_{+}^{2}\times\RR^{d}}  \mathcal{F} \varphi(s,\xi) \overline{ \mathcal{F} \psi(t,\xi)}\gamma(s-t) \mu(d\xi) \, dsdt,
\end{eqnarray}
where $\mathcal{F} \varphi$ refers to the Fourier transform with respect to the space variable only.
 The mapping $\varphi \rightarrow W(\varphi)$ defined in $\mathcal{D}((0,\infty)\times\RR^d)$  extends to a linear isometry between
$\mathcal{H}$  and the Gaussian space
spanned by $W$. We will denote this isometry by
\begin{equation*}
W(\phi)=\int_0^{\infty}\int_{\RR^d}\phi(t,x)W(dt,dx)
\end{equation*}
for $\phi \in \mathcal{H}$.
Notice that if $\phi$ and $\psi$ are in
$\mathcal{H}$, then
$\EE \left[ W(\phi)W(\psi)\right] =\langle\phi,\psi\rangle_{\mathcal{H}}$.

\smallskip

We shall make a standard assumption on the spectral measure $\mu$, which will prevail until the end of the paper.
\begin{hypothesis}\label{hyp:mu}
The measure $\mu$ satisfies the following integrability condition:
\begin{equation} \label{mu1}
\int_{\RR^{d}}\frac{\mu(d\xi)}{1+|\xi|^{2}} <\infty.
\end{equation}
\end{hypothesis}

Now we state some basic facts about Malliavin calculus. For a detailed account on this theory, we refer to \cite{Nua}.
We will denote by $D$ the Malliavin derivative. That is, if $F$ is a smooth and cylindrical
random variable of the form
\begin{equation*}
F=f(W(\phi_1),\dots,W(\phi_n))\,,
\end{equation*}
with $\phi_i \in \mathcal{H}$, $f \in C^{\infty}_p (\RR^n)$ (namely, $f$ and all
its partial derivatives have polynomial growth), then $DF$ is the
$\mathcal{H}$-valued random variable defined by
\begin{equation*}
DF=\sum_{j=1}^n\frac{\partial f}{\partial
x_j}(W(\phi_1),\dots,W(\phi_n))\phi_j\,.
\end{equation*}
The operator $D$ is closable from $L^p(\Omega)$ into $L^p(\Omega;
\mathcal{H})$  for any $p\ge 1$ and we define the Sobolev space $\mathbb{D}^{1,p}$ as
the closure of the space of smooth and cylindrical random variables
under the norm
\[
\|F\|_{1,p}=\left(\EE[ |F|^p]+\EE[\|DF\|^p_{\mathcal{H}}]\right)^{\frac{1}{p}}\,.
\]
We denote by $\delta$ the adjoint of the derivative operator given
by the duality formula
\begin{equation}\label{dual}
\EE \left[ \delta (u)F \right] =\EE \left[ \langle DF,u
\rangle_{\mathcal{H}}\right] ,
\end{equation}
for any $F \in \mathbb{D}^{1,2}$ and any element $u \in L^2(\Omega;
\mathcal{H})$ in the domain of $\delta$. The operator $\delta$ is
also called the {\it Skorohod integral} because in the case of the
Brownian motion, it coincides with an extension of the It\^o
integral introduced by Skorohod.    We will make use of the notation
\[
\delta(u) =\int_0^t \int_{\mathbb{R}^d}  u(s,y) \delta W_{s,y}.
\]

If $F\in \mathbb{D}^{1,2}$ and $\phi$ is an element of $%
\mathcal{H}$, then $F\phi$ is Skorohod
integrable and, by definition, the
Wick product equals the Skorohod integral of $Fh$, that is, 
\begin{equation}
\delta (F\phi)=F\diamond W(\phi).  \label{Wick}
\end{equation}%
In view of this definition, the mild solution to equation  (\ref{SHE}) will be formulated  below  in terms of the Skorohod integral.

Next we give a short account of Wiener chaos expansion. For any integer $n\geq 0$ we denote by $\mathbf{H}_n$ the $n$th Wiener chaos of $W$. We recall that $\mathbf{H}_0$ is simply  $\RR$ and for $n\ge 1$, $\mathbf {H}_n$ is the closed linear subspace of $L^2(\Omega)$ generated by the random variables $\{ H_n(W(h)), h \in \mathcal{H}, \|h\|_{\mathcal{H}}=1 \}$, where $H_n$ is the $n$th Hermite polynomial. Then we will have the orthogonal decomposition
\begin{equation}
L^2(\Omega)=\oplus _{n=0}^{\infty} \mathbf{H}_n\,.
\end{equation}
For each $n \geq 0$, we will denote by $J_n$ the orthogonal projection on the $n$th Wiener chaos. Consider the one-parameter semigroup $\{T_t, t\geq 0\}$ of contraction operators on $L^2(\Omega)$ defined by
\begin{equation}
T_t(F)=\sum_{n=0}^{\infty} e^{-nt}J_n F\,,
\end{equation}
which is called the Ornstein-Uhlenbeck semigroup. The following property of $T_t$ is taken from \cite{Nua1}.

\begin{proposition}\label{prop: OU semi}
For any $p > 1$, if $F \in L^p (\Omega)$, then $T_t F \in \mathbb{D}^{1,p}$ for any $t >0$ and we also have
\begin{equation}
\lim_{t \to 0} \|T_t F -F\|_{1,p}=0\,.
\end{equation}
\end{proposition}



We are ready to give the definition of  mild  solution to equation \eref{SHE}. We denote by $p_{t}(x)$ the $d$-dimensional heat kernel $p_{t}(x)=(2\pi t)^{-d/2}e^{-|x|^2/2t} $, for any $t > 0$, $x \in \R^d$. For each $t\ge 0$, let $\mathcal{F}_t$ be the $\sigma$-field generated by the random variables $W(\varphi)$, where $\varphi$ has support in $[0,t ]\times \mathbb{R}^d$. We say that a random field $u(t,x)$ is adapted if for each $(t,x)$ the random variable $u(t,x)$ is $\mathcal{F}_t$-measurable. Then we have the following definition.

\begin{definition}\label{def}
An adapted   random field $u=\{u(t,x), t \geq 0, x \in
\mathbb{R}^d\}$ such that $\EE [ u^2(t,x)] < \infty$ for all $(t,x)$ is
a mild solution to equation \eref{SHE} with bounded initial condition $u_0$, if for any $(t,x) \in [0,
\infty)\times \mathbb{R}^d$, the process $\{p_{t-s}(x-y)u(s,y){\bf
1}_{[0,t]}(s), s \geq 0, y \in \mathbb{R}^d\}$ is Skorohod
integrable, and the following equation holds
\begin{equation}\label{eq:sko-mild}
u({t,x})=p_t u_0(x)+\int_0^t\int_{\mathbb{R}^d}p_{t-s}(x-y)u(s,y) \, \delta
W_{s,y} \quad a.s.
\end{equation}
\end{definition}

The following theorem about the existence and uniqueness of the solution to equation \eref{SHE} is taken from \cite{HHNT}.

 \begin{theorem}\label{thmSk1}
Suppose that   $\mu$ satisfies Hypothesis  \ref{hyp:mu}
and $\gamma$ is locally integrable. Then equation \eref{SHE} admits a
unique mild solution in the sense of Definition \ref{def}.
\end{theorem}

The next lemma states the non-negativity of the solution and will be used in the next section.

 \begin{lemma}\label{lemma:sol nonnegative}
Assume that  $\mu$ satisfies Hypothesis  \ref{hyp:mu}
and $\gamma$ is locally integrable. If the initial condition $u_0$ is nonnegative, then for each $(t,x)\in [0,\infty)\times \RR^d$, $u(t,x)\geq 0 $ a.s.
 \end{lemma}
 \begin{proof}
 We will follow the procedure in Section 3.2, \cite{HHNT}.  For any $\delta > 0$, we define the function $\varphi_{\delta}(t)=\frac{1}{\delta}{\bf 1}_{[0,\delta]}(t)$ for $t \in \R$.  Then, $\varphi_{\delta}(t)p_{\varepsilon}(x)$  provides an approximation of the Dirac delta function $\delta_0(t,x)$ as $\varepsilon$ and $\delta$ tend to zero. Define
 \begin{equation}\label{eq9}
u^{\varepsilon,\delta}(t,x)=\EE_B \left[ \exp \left(  W (
A_{t,x}^{\varepsilon,\delta})-\frac{1}{2}\alpha^{\varepsilon,\delta}_{t,x}\right)
\right]\,,
\end{equation}
where
\begin{equation}  \label{m3}
A_{t,x}^{\varepsilon,\delta}(r,y)=\frac 1\delta
\left(\int_0^{\delta \wedge (t-r)}
p_{\varepsilon}(B_{t-r-s}^x-y)ds\right) \mathbf{1}_{[0,t]} (r),
\quad\text{and}\quad
\alpha^{\varepsilon,\delta}_{t,x}=\|A^{\varepsilon,\delta}_{t,x}\|^2_{\mathcal{H}},
\end{equation}
for  a standard $d$-dimensional Brownian motion $B$ independent of
$W$.  Then it is obvious from the definition of $u^{\varepsilon,\delta}(t,x)$ that $u^{\varepsilon,\delta}(t,x) > 0$ a.s. for each $(t,x)$, and from Theorem 3.6 in \cite{HHNT} and its proof, we see that for each $F \in \mathbb{D}^{1,2}$ and $(t,x)$, $\EE (F u^{\varepsilon,\delta}(t,x))$ converges to $\EE (F u(t,x))$. Now for each fixed $(t,x)$, we take $F=T_s {\bf 1}_{\{u(t,x)<0\}}$, from Proposition \ref{prop: OU semi} we know that such $F$ is in $\mathbb{D}^{1,2}$, for each $s > 0$. So we have
\begin{equation*}
\EE \left( \left(T_s {\bf 1}_{\{u(t,x)<0\}}\right) u(t,x) \right)=\lim_{\delta,\varepsilon \to 0} \EE \left( \left(T_s {\bf 1}_{\{u(t,x)<0\}}\right) u^{\varepsilon,\delta}(t,x) \right) \geq 0\,,
\end{equation*}
then letting $s$ go to $0$ we obtain  by Proposition \ref{prop: OU semi}
\begin{equation*}
\EE \left( {\bf 1}_{\{u(t,x)<0\}} u(t,x) \right) \geq 0\,,
\end{equation*}
which shows that $u(t,x)\geq 0$ a.s.
 \end{proof}

The next result concerning the moment formula for the solution is taken from \cite{HHNT}, see also \cite{Con}, where $\gamma$ is the Dirac delta function $\delta$.

 \begin{theorem}\label{thm:mom}
Suppose $\gamma$ is locally integrable and $\mu$ satisfies
Hypothesis  \ref{hyp:mu}. Let $u(t,x)$ be the solution to equation \eref{SHE}. Then for any  integer $p\geq 2$
\begin{equation}\label{eq:mom}
\EE   u^p(t,x)  =\EE_B \left[ \prod_{i=1}^p u_0(B_t^i+x) \exp\left( \lambda^2\sum_{1 \leq i < j
\leq p}\int_0^t \int_0^t \gamma (s-r)\Lambda(B_s^i-B_r^j)ds
dr\right)\right]\,,
\end{equation}
where $\{B^j, \, j=1,\dots, p\}$  is a family of $d$-dimensional  independent standard Brownian motions independent of $W$.
\end{theorem}

\section{Main results}

We need first to introduce some notation. If $\mu$ is a measure satisfying Hypothesis \ref{hyp:mu},  for any  real number $N>0$, we define
\begin{equation}\label{eq:C_N D_N}
C_N = \int_{|\xi| > N}\frac{\mu(d\xi)}{|\xi|^2}, \quad \text{and} \quad D_N= \mu\left\{ \xi: |\xi| \leq N \right\}\,.
\end{equation}
On the other hand, if $\gamma$ is locally integrable, we set
\begin{equation}\label{Gamma_t}
\int_0^t \gamma(s)ds = \Gamma_t\,.
\end{equation}

\begin{theorem}\label{thm:cor upper}
Let $u(t,x)$ be the solution to equation \eref{SHE}  driven by a noise $W$ with covariance  structure  \eref{cov}. Assume that  $u_0$ is non-negative and supported in the ball  $B_M=\{ x\in \RR^d: |x| \le M \}$.  Assume that $\mu$  satisfies Hypothesis \ref{hyp:mu} and $\gamma $ is locally integrable.
Set $\theta_t = \sqrt{D_{N_t} C_{N_t}^{-1}}$, where
\begin{equation}\label{N_t}
N_t = \inf \left\{N \geq 0: C_N \leq \frac{(2\pi)^d}{32 (p-1) \lambda^2 \Gamma_t} \right\}\,.
\end{equation}
Then, for any integer $p\ge 2$, we have
\begin{equation}
\bar{\nu}(p):= \inf \left \{ \varrho>0:  \limsup _{t \to \infty} \frac{1}{t\theta^2_t} \sup_{|x|\geq \varrho t\theta_t} \log \EE u^p(t,x) < 0 \right\} \le 1\,.
\end{equation}

\end{theorem}
\begin{proof}
Using the moment formula (\ref{eq:mom}), together with Cauchy-Schwartz  inequality,  we can write for any integer  $p\ge 2$
\begin{eqnarray*}
\EE u^p(t,x)& =& \EE_B \left( \prod_{i=1}^p u_0(x+B_t^i) \exp \left (\lambda^2 \sum_{1 \leq i < j \leq p}   \int_0^t \int_0^t \gamma(s-r) \Lambda(B_s^i-B_r^j) ds dr \right) \right)\\
&\leq& \EE_B \left( \prod_{i=1}^p {\bf 1}_{B_M}(x+B_t^i) \exp \left (\lambda^2 \sum_{1 \leq i < j \leq p}   \int_0^t \int_0^t \gamma(s-r) \Lambda(B_s^i-B_r^j) ds dr \right) \right) \|u_0\|^p_{\infty}\\
&\leq&\left( \EE_B   \exp \left ( 2 \lambda^2\sum_{1 \leq i < j \leq p}   \int_0^t \int_0^t \gamma(s-r) \Lambda(B_s^i-B_r^j) ds dr \right) \right)^{\frac{1}{2}} \\
&& \times  \left( \EE {\bf 1}_{B_M}(x+B^1_t) \right)^{\frac{p}{2}}\|u_0\|^p_{\infty}\,.
\end{eqnarray*}
Note that in the above expression, the first expectation in the  last inequality is exactly the $p$th moment of the solution   to equation \eref{SHE} (denoted by $v(t,x)$) with noise $W$ having a covariance functional with parameters   $\gamma$ and { ${2}  \Lambda$} respectively, and with initial condition $1$. From \cite{HHNT} we see that $v(t,x)$ admits the chaos expansion 
\begin{equation*}
v(t,x)=\sum_{n=0}^{\infty} I_n(f_n(\cdot,t,x)) \,,
\end{equation*}
 where  for $n \geq 1$, the kernel $f_n$  is given by
 \begin{equation*}
 f_n(s_1,x_1, \dots, s_n, x_n, t,x) = \frac{1}{n!} p_{t-s_{\sigma(n)}}(x-x_{\sigma(n)}) \cdots p_{s_{\sigma(2)}-s_{\sigma(1)}}(x_{\sigma(2)}-x_{\sigma(1)})\,.
 \end{equation*}
 In the above expression, $\sigma$ is the permutation of $\{1,2,\dots,n\}$ such that $s_{\sigma(1)}< s_{\sigma(2)} < \cdots < s_{\sigma(n)} <  s_{\sigma(n+1)}=t$.  We have
\begin{equation*}
\EE [I_n(f_n(\cdot,t,x))^2]=n! \|f_n(\cdot, t,x)\|^2_{\mathcal{H}_1^{\otimes n}}\,,
\end{equation*}
where $\mathcal{H}_1$ denotes the norm introduced in \eref{innprod}, but with $\Lambda$ replaced by {${2}   \Lambda$}. By the hypercontractivity property, we have
\begin{eqnarray*}
\|I_n(f_n(\cdot,t,x))\|_{L^p(\Omega)}\leq (p-1)^{\frac{n}{2}}\|I_n(f_n(\cdot,t,x))\|_{L^2(\Omega)}\,.
\end{eqnarray*}
Therefore,  have the $L^p$ norm of the $v(t,x)$ is bounded as follows
\begin{eqnarray*}
\|v(t,x)\|_{L^p(\Omega)}&\leq& \sum_{n=0}^{\infty} \|I_n(f_n(\cdot,t,x))\|_{L^p(\Omega)} \leq \sum_{n=0}^{\infty}(p-1)^{\frac{n}{2}}\|I_n(f_n(\cdot,t,x))\|_{L^2(\Omega)}\\
&=&\sum_{n=0}^{\infty}(p-1)^{\frac{n}{2}} \sqrt{n!} \|f_n(\cdot,t,x)\|_{\mathcal{H}_1^{\otimes n}}\,.
\end{eqnarray*}
Then using Fourier transform as in \cite{HHNT} we have
\begin{eqnarray*}
n!\|f_n(\cdot,t,x)\|^2_{\mathcal{H}_1^{\otimes n}} &\leq& \frac{(2\lambda^2)^n n!}{(2\pi)^{nd}} 
\int_{[0,t]^{2n}}\int_{\RR^{nd}}  |\mathcal{F}f_n(s,\cdot,t,x)(\xi)|^2\mu(d\xi) \prod_{i=1}^n \gamma(s_i-r_i)ds dr \\
&\leq&  \frac{ (4 \lambda^2\Gamma_t)^n n!}{(2\pi)^{nd}} 
\int_{[0,t]^{n}}\int_{\RR^{nd}}  |\mathcal{F}f_n(s,\cdot,t,x)(\xi)|^2\mu(d\xi) ds\\
&\leq& \frac{ (4 \lambda^2 \Gamma_t)^n}{(2\pi)^{nd}} \int_{\RR^{nd}} \int_{T_n(t)} \prod_{i=1}^n e^{-(s_{i+1}-s_i)|\xi_i|^2} ds \mu(d\xi)\,,
\end{eqnarray*}
where $T_n(t)$ denotes the simplex $\{0\leq s_1 \leq s_2 \leq \dots \leq s_n\leq t\}$,  $\mu(d\xi) = \prod_{i=1}^n \mu(d\xi_i)$ and $ds$ is defined similarly.  Then with the change of variable $s_{i+1}-s_i = w_i$ and by Lemma \ref{lemmma1}  below applied to $N=N_t$, we obtain
\begin{eqnarray*}
n!\|f_n(\cdot,t,x)\|^2_{\mathcal{H}_1^{\otimes n}} &\leq&   \frac{ (4 \lambda^2 \Gamma_t)^n}{(2\pi)^{nd}} \int_{S_{t,n}} \int_{\RR^{nd}} \prod_{i=1}^n e^{-w_i|\xi_i|^2} \mu(d\xi) dw\\
&\leq&  \frac{ (4 \lambda^2 \Gamma_t)^n}{(2\pi)^{nd}} \sum_{k=0}^n {n \choose k} \frac{t^k}{k!} D_{N_t}^k C_{N_t}^{n-k}\\
&\leq& \left(\frac{8 \lambda^2  \Gamma_t C_{N_t}}{(2\pi)^d} \right)^n \sum_{k=0}^{\infty} \frac{t^k D_{N_t}^k C_{N_t}^{-k}}{k!}\\
&=& \left(\frac{8 \lambda^2 \Gamma_t C_{N_t}}{(2\pi)^d} \right)^n e^{t D_{N_t} C_{N_t}^{-1} }\,,
\end{eqnarray*}
where $S_{t,n}=\{(w_1, \dots, w_n)\in [0,  \infty)^n: w_1 +\cdots +w_n \le t\}$. 
Thus
\begin{eqnarray*}
\|v(t,x)\|_{L^p(\Omega)} &\leq& \sum_{n=0}^{\infty} \left(\frac{8 \lambda^2 \Gamma_t C_{N_t}}{(2\pi)^d} \right)^{\frac{n}{2}} (p-1)^{\frac{n}{2}}e^{\frac{1}{2} t D_{N_t} C_{N_t}^{-1} }
\leq 2 e^{\frac{1}{2}t D_{N_t} C_{N_t}^{-1}}\,,
\end{eqnarray*}
where the last inequality comes from the definition of ${N_t}$.  Thus we obtain from Lemma \ref{lem:heat kernel int} below
\begin{eqnarray*}
\EE u(t,x)^p \leq 2^{\frac p2}  e^{\frac{p}{4}t D_{N_t} C_{N_t}^{-1}} \frac{1}{(2\pi t)^{\frac{dp}{4}}} e^{-\frac{|x|^2p}{4 t (\varkappa +1)}} e^{\frac{M^2p}{4 t \varkappa}} \omega_d^{\frac{p}{2}} M^{\frac{dp}{2}} \| u_0\|^p_\infty
\end{eqnarray*}
for any $\varkappa > 0$. 
As a consequence,  if we want 
\begin{eqnarray*}
\limsup_{ t\to \infty} \frac{1}{t \theta_t^2} \log \sup_{|x| \geq \varrho t \theta_t} \EE u(t,x)^p\leq \frac{p}{4}- \frac{p \varrho^2}{4(\varkappa+1)}< 0\,,
\end{eqnarray*}
we need $\varrho > \sqrt{\varkappa +1}$. Letting $\varkappa \to 0$   we conclude that $\bar{\nu}(p)\leq 1$. 
\end{proof}

Section 6 in \cite{HHNT} gives the moment upper bounds for some specific choices of $\gamma$ and $\Lambda$, assuming the initial condition is a bounded function. Actually the proof of Theorem \ref{thm:cor upper} above also gives a general upper bound for the $p$th moment,  stated in the following corollary.
\begin{corollary}  
Let $u(t,x)$ be the solution to equation \eref{SHE} with a bounded nonnegative initial condition. Let $D_N, C_N$ be defined as in \eref{eq:C_N D_N} and $N_t$ be defined as in \eref{N_t}. Then we have the moment upper bound
\begin{equation}
\EE u^p(t,x) \leq C^p \exp \left(C p t D_{N_t} C^{-1}_{N_t} \right)\,,
\end{equation}
for some constant $C$ independent of $p$ and $t$. 
\end{corollary}

The next lemma is used in the proof of Theorem \ref{thm:cor upper}. For a proof, see Lemma 3.3 in \cite{HHNT}.
\begin{lemma}\label{lemmma1}
Let $\mu$ satisfy Hypothesis \ref{hyp:mu}. For any $N
> 0$ let $D_N$ and $C_N$ be given by (\ref{eq:C_N D_N}). Then
we have
\[
 \int_{\RR^{nd}}\int_{S_{t,n}}e^{- \sum_{i=1}^n w_i |\xi_i|^2}
 dw  \mu(d\xi)
\leq\sum_{k=0}^n {n \choose k}
\frac{t^k}{k!}D_N^k C_N^{n-k}\,.
\]
\end{lemma}


The next result is a lower bound for the lower intermittency front, when  $\Lambda$ is bounded below by the Riesz kernel.

\begin{theorem}\label{thm:cor lower}
Let $u(t,x)$ be the solution to equation \eref{SHE} with  nonnegative initial condition $u_0$ being uniformly bounded away from $0$ in the ball $B_{ M}$. 
Assume that
\begin{equation}
\Lambda(x)\geq C_{\Lambda} |x|^{-\beta}, \ \forall |x|\leq R,  \ \text{for some } R >0
\end{equation}
with $0 \leq \beta < 2\wedge d$. Suppose that
\begin{equation}  \label{w2}
\lim_{t \rightarrow \infty} \Gamma_t =\infty.
\end{equation}
Fix $\delta \in (0,1)$ and set  $\eta_t = \Gamma^{\frac{1}{2-\beta}}_{ t\delta^2}$. Define
\begin{equation}  \label{39}
\underline{\nu}(p):= \sup \left \{ \varrho>0:  \limsup _{t \to \infty} \frac{1}{t \eta^2_t} \sup_{{|x|\geq \varrho t \eta_t }} \log E (|u(t,x)|^{p}) > 0 \right\} \,.
\end{equation}
Then we have
\begin{equation}
\underline{\nu}(p) \geq \sqrt{C_{\beta,\delta}} \lambda^{\frac{2}{2-\beta}}  (p-1)^{\frac{ 1}{2-\beta}}\,,
\end{equation}
where
\begin{equation*}
C_{\beta,\delta}=2 \left[\left(\frac{\beta}{2} \right)^{\frac{\beta}{2-\beta}} -  \left(\frac{\beta}{2} \right)^{\frac{2}{2-\beta}}\right]  (1-\delta)^{ \frac 2{2-\beta}} {\delta } j_{\nu}^{\frac{-2\beta}{2-\beta}} C_{\Lambda}^{\frac{2}{2-\beta}}\sqrt{2(1-\delta)} \,,
\end{equation*}
and $j_{\nu}$ denotes the smallest positive zero of the Bessel function $J_{\nu}(x)$ of index $\nu = \frac{d-2}{2}$. 
\end{theorem}

\begin{proof}
Note that using the change of variable $s \rightarrow \frac{u+v}{2}$ and $r \rightarrow \frac{v-u}{2}$ we have 
\begin{eqnarray*}
\int_0^t \int_0^t \gamma(s-r) ds dr
&=&\int_{-t} ^t \int_u^{2t-u}\gamma(u)dv du
=4 \int_0^t \gamma(u)(t-u) du \\
&
\geq&  4 (1- \delta) t \int_0^{t\delta}\gamma(u) du = 4(1-\delta) t \Gamma_{t\delta }\,.
\end{eqnarray*}
Let $u_0(x) \geq C_{u_0} {\bf 1}_{M}(x)$.  Then using the moment formula for the solution $u(t,x)$ as before,
\begin{eqnarray*}
\EE  u^p(t,x) & =& \EE_B \left( \prod_{i=1}^p u_0(x+B_t^i) \exp \left (\lambda^2 \sum_{1 \leq i < j \leq p}   \int_0^t \int_0^t \gamma(s-r) \Lambda(B_s^i-B_r^j) ds dr \right) \right)\\
&\geq & \EE_B \left( \prod_{i=1}^p u_0(x+B_t^i) \exp \left ( \lambda^2 \sum_{1 \leq i < j \leq p}   \int_0^{t\delta} \int_0^{t\delta} \gamma(s-r) \Lambda(B_s^i-B_r^j) ds dr \right) \right) \\
&\geq&C_{u_0}^p P \left (\sup_{1\leq i \leq p}|B^i_t+x|\leq M,  \,  \sup_{0\leq s \leq t \delta, 1\leq i \leq p} |B^i_s|\leq \varepsilon \right)  \\
&& \times \exp \left ( 2\lambda^2  p(p-1)  (1-\delta) { \delta t } \Gamma_{t\delta^2}  C_{\Lambda} |2\varepsilon|^{-\beta}\right)\\
&=&C_{u_0}^p  P \left (|B^0_t+x|\leq M,  \,  \sup_{0\leq s \leq t\delta} |B^0_s|\leq {\varepsilon} \right)^p \\
&&\times  \exp \left (2\lambda^2 p(p-1) (1-\delta){ \delta t }    \Gamma_{ t\delta^2}  C_{\Lambda} |2\varepsilon|^{-\beta}\right)\,,
\end{eqnarray*}
where $B^0_s$ is a standard Brownian motion, $\varepsilon$ is a positive number satisfying $\varepsilon <\frac R2$,  which will be chosen later. In order to  estimate of the above probability, notice that
\begin{eqnarray*}
&&P \left (|B^0_t+x|\leq M, \sup_{0\leq s \leq   t\delta} |B^0_s|\leq \varepsilon \right)\\
&=&\EE \left ( \EE\left \{ {\bf 1}_{\{|B^0_t-B^0_{t\delta }+x+B^0_{t\delta}|\leq M, \,  \sup_{0 \leq s \leq  t\delta}|B^0_s| \leq \varepsilon\}}\Big|\mathcal{G}_{t\delta }\right\}\right)\,,
\end{eqnarray*}
where $\mathcal{G}_t$ is the filtration generated by $\{B^0_s: 0 \leq s \leq t\}$. Then we can choose $\varepsilon \leq \frac{M}{2}$(the specific choice of $\varepsilon$ will be given below) and invoke Lemma \ref{lem:heat kernel int} to get, for $t$ large enough,
\begin{eqnarray*}
&&P \left (|B^0_t+x|\leq  M, \sup_{0\leq s \leq  t\delta } |B^0_s|\leq \varepsilon \right) \geq \EE \left ( \EE\left \{ {\bf 1}_{\{|B^0_t-B^0_{ t\delta}+x|\leq \frac{M}{2}}\Big|\mathcal{G}_{t\delta }\right\} {\bf 1} _{\{\sup_{0 \leq s \leq  t\delta}|B^0_s| \leq \varepsilon\}}\right)\\
&&\qquad \geq P \left\{|B^0_{t(1-\delta) }+x|\leq \frac{M}{2} \right\} P\left \{ \sup_{0 \leq s \leq  1} |B^0_s|\leq \frac{\varepsilon}{\sqrt{t \delta}} \right\} \\
&& \qquad  \geq  C  \frac{\omega_d M^d}{(t(1-\delta))^{\frac{d}{2}}} e^{-\frac{M^2(1+\frac{1}{\varkappa})}{8t(1-\delta)}} e^{-\frac{(\varkappa+1)|x|^2}{2t(1-\delta)}}  e^{-\frac{j_{\nu}^2  t \delta}{2 \varepsilon^2}}\,,
\end{eqnarray*}
where    $C$ is a universal constant. The last inequality follows from the small ball probability estimate (see Theorem 1 in \cite{Shi})
\begin{equation}
P \left \{ \sup_{0 \leq s \leq 1} |B_s| \leq \varepsilon \right\} \sim e^{-\frac{j_{\nu}^2}{2 \varepsilon^2}}, \quad \text{as} \quad \varepsilon \to 0\,.
\end{equation}
Then we obtain
\begin{eqnarray*}
\EE u^p(t,x) &\geq &   (C C_{u_0})^p (t(1-\delta)) ^{-\frac{dp}{2}} \exp\Bigg(-\frac{M^2(1+\frac{1}{\varkappa})}{8t(1-\delta)}-\frac{p (\varkappa +1)|x|^2}{2t(1-\delta)} -\frac{pj_{\nu}^2  t\delta}{2 \varepsilon^2} \\
&& +2  \lambda^2 p(p-1) (1-\delta) { \delta t \Gamma_{t\delta^2 } }C_{\Lambda} |2\varepsilon|^{-\beta} \Bigg)\,.
\end{eqnarray*}
We apply Lemma \ref{lem:max} with
\begin{equation*}
A = 2^{-\beta+1} \lambda^2 p(p-1) (1-\delta) { \delta t \Gamma_{t\delta^2} }C_{\Lambda}\, \quad \text{and} \quad B= \frac{j_{\nu}^2  t  \delta p}{2}
\end{equation*}
to maximize the right hand side of the above inequality, by choosing
\begin{equation*}
\varepsilon = \left ( \frac{{ 2^{\beta-1} j_{\nu}^2 }}{\beta \lambda^2 (p-1)  (1-\delta)\Gamma_{t\delta^2} C_{\Lambda}}\right)^{\frac{1}{2-\beta}}.
\end{equation*}
In this way we obtain
\begin{eqnarray*}
&\EE u^p(t,x) \geq  ( CC_{u_0})^p (t(1-\delta))^{-\frac{dp}{2}} \exp \left (-\frac{M^2(1+\frac{1}{\varkappa})}{8t(1-\delta)}-\frac{p (\varkappa +1)|x|^2}{2t(1-\delta)} + C_{\beta,\delta} \lambda^{\frac{4}{2-\beta}} p (p-1)^{\frac{2}{2-\beta}} t \Gamma_{t\delta^2}^{\frac{2}{2-\beta}} \right)\,.
\end{eqnarray*}
We remark  condition (\ref{w2}) implies that  $\varepsilon$ chosen above tends to $0$ as $t \to \infty$, thus the small ball estimate used above works  for $t$ large enough, in such a way that $\varepsilon <\frac M2$. 
If we want
\begin{eqnarray*}
\limsup_{t \to \infty} \frac{1}{t\eta_t^2}\log \sup_{ { |x|\geq \varrho t\eta_t  }  } \EE u^p(t,x) 
 \geq \frac{-p (\varkappa +1)\varrho^2}{2(1-\delta)} + C_{\beta,\delta} \lambda^{\frac{4}{2-\beta}}  p^{\frac{4-\beta}{2-\beta}} > 0\,.
\end{eqnarray*}
we need
\begin{equation*}
\varrho < \frac{\sqrt{C_{\beta,\Lambda}}}{\sqrt{\varkappa +1}} \lambda^{\frac{2}{2-\beta}}  (p-1)^{\frac{ 1}{2-\beta}}\,,
\end{equation*}
Letting $\varkappa \to 0$ and invoking Lemma \ref{lemma:sol nonnegative} we conclude that
\begin{equation*}
\underline{\nu}(p)\geq \sqrt{C_{\beta,\delta}} \lambda^{\frac{2}{2-\beta}}  (p-1)^{\frac{ 1}{2-\beta}} \,.
\end{equation*}
The theorem is proved.
\end{proof}

The next two lemmas are used in the proof of previous theorems. 
\begin{lemma}\label{lem:heat kernel int}
For any positive $M$ and $\varkappa$ we have 
\begin{equation}
\frac{1}{(2\pi t)^{\frac{d}{2}}} e^{-\frac{(\varkappa+1)|x|^2}{2t}} e^{-\frac{M^2(1+\frac{1}{\varkappa})}{2t}}\omega_d M^d \leq \int_{|y|\leq M} \frac{1}{(2\pi t)^{\frac{d}{2}}} e^{-\frac{|y-x|^2}{2t}} dy \leq \frac{1}{(2\pi t)^{\frac{d}{2}}} e^{-\frac{|x|^2}{2 t (\varkappa +1)}} e^{\frac{M^2}{2 t \varkappa}} \omega_d M^d\,,
\end{equation}
where $\omega_d$ is the volume of the unit ball in $\RR^d$. 
\end{lemma}
\begin{proof}
We have the simple fact that 
\begin{eqnarray*}
|y-x|^2 \leq |y|^2 + 2 |x||y|+ |x|^2 \leq |y|^2 + \frac{|y|^2}{\varkappa} + \varkappa |x|^2 + |x|^2 = (1+ \frac{1}{\varkappa}) |y|^2 + (\varkappa +1)|x|^2\,,
\end{eqnarray*}
and from this we deduce the reverse inequality 
\begin{equation*}
|x-y|^2 \geq \frac{|x|^2}{\varkappa+1} - \frac{|y|^2}{\varkappa}\,,
\end{equation*}
from these two inequalities the proof is done by elementary calculations. 
\end{proof}

\begin{lemma}\label{lem:max}
Let $A,B > 0$ and $0 < \beta < 2$. Then the function
\begin{equation*}
f(x)= Ax^{-\beta}-Bx^{-2}
\end{equation*}
attains its maximum at $x= \left ( \frac{2B}{\beta A}\right)^{\frac{1}{2-\beta}}$, and the maximum value equals
\[
\left(\left ( \frac{\beta}{2}\right)^{\frac{\beta}{2-\beta}} - \left(\frac{\beta}{2} \right)^{\frac{2}{2-\beta}}\right) A^{\frac{2}{2-\beta}}B^{\frac{-\beta}{2-\beta}}.
\]

\end{lemma}

\begin{remark}
If condition (\ref{w2}) does not hold, that is, the limit is finite $\Gamma_\infty$ (which happens, for instance, if $\gamma$ is a Dirac function), then, we need the following additional condition on $M$:
\begin{equation}\label{restriction on M}
R \wedge M \geq 2\left ( \frac{{  2^{\beta-1} j_{\nu}^2  }}{\beta \lambda^2 (p-1)  (1-\delta)\Gamma_{\infty} C_{\Lambda}}\right)^{\frac{1}{2-\beta}}.
\end{equation}
\end{remark}

\begin{remark}
When $\gamma$ is the Dirac delta function,  the noise $\dot W$ is white in time and correlated in space and Theorems \ref{thm:cor upper} and \ref{thm:cor lower} still hold with $\Gamma_t=  \frac{1}{2}$. The appearance of the functions $\theta_t$ and $\eta_t$ in Theorems \ref{thm:cor upper} and \ref{thm:cor lower} come from the time covariance of the noise. In the case when $\Gamma_t \to \infty$ as $t \to \infty$ (for instance, when $\gamma(t)=|t|^{-\alpha}$ with $0< \alpha < 1$, $\Gamma_t = \frac{t^{1-\alpha}}{1-\alpha}$),  the restriction on $M$ \eref{restriction on M} is not needed. 
\end{remark}

In the case where $\Lambda(x)=|x|^{-\beta}$ with $0 < \beta < 2\wedge d$, we obtain the following more precise result concerning the upper bound of the intermittency front.  In this case, 
  the  function $\theta_t$ defined in Theorem \ref{thm:cor upper} can be replaced by  $\Gamma_t^{\frac 1 {2-\beta}} $. Notice that this function coincides with the function $\eta_t$ in Theorem  \ref{thm:cor lower} except for the factor $\delta^2$. If we let $\delta$ tend to $1$, then the 
lower bound in  (\ref{39})  tends to zero.

\begin{proposition}\label{prop: Riesz upper}
Assume that  $u_0$ is non-negative and supported in the ball  $B_M$. 
Let $\gamma$ be a locally integrable, positive and positive definite function, and $ \Lambda(x)=|x|^{-\beta}$, assume $0 < \beta < 2\wedge d$. Set $\vartheta _t = \Gamma_t^{\frac{1}{2-\beta}}$, where $\Gamma_t$ is defined in \eref{Gamma_t}.  Define 
\begin{equation}\label{nu(p) up}
\bar{\upsilon}(p):= \inf \left \{ \varrho>0:  \limsup _{t \to \infty} \frac{1}{t\vartheta^2_t} \sup_{|x|\geq \varrho t\vartheta_t} \log \EE u^p(t,x) < 0 \right\}\,,
\end{equation}
then 
\begin{equation*}
\bar{\upsilon}(p)\leq 2\sqrt{2} \left( \frac{ \Gamma\left(\frac{d-\beta}{2}\right)\Gamma\left(1-\frac{\beta}{2} \right)}{ \Gamma\left(\frac{d}{2}\right) } \right)^{\frac{1}{2-\beta}} (p-1)^{\frac{1}{2-\beta}} \lambda^{\frac{2}{2-\beta}}\,.
\end{equation*}
\end{proposition}
\begin{proof}
We will follow the notations and the same calculations used in the proof of Theorem \ref{thm:cor upper}. We have
\begin{equation*}
n!\|f_n(\cdot,t,x)\|^2_{\mathcal{H}_1^{\otimes n}} \leq  \frac{ (4 \lambda^2 \Gamma_t)^n}{(2\pi)^{nd}} \int_{\RR^{nd}} \int_{T_n(t)} \prod_{i=1}^n e^{-(s_{i+1}-s_i)|\xi_i|^2} ds \mu(d\xi)\,. 
\end{equation*}
Since $\Lambda(x)=|x|^{-\beta}$, its Fourier transform is given by   (see, e.g., Chapter 5 in \cite{Ste})
\begin{eqnarray*}
\mu(d\xi) = \frac{\pi^{\frac{d}{2}}2^{d-\beta} \Gamma\left(\frac{d-\beta}{2}\right)}{\Gamma(\frac{\beta}{2})} |\xi|^{\beta-d} d\xi := \Lambda_{\beta} |\xi|^{\beta-d}d\xi\,.
\end{eqnarray*}
Using polar coordinates we can compute the integral 
\begin{equation*}
\int_{\RR^d} e^{-|\eta|^2}|\eta|^{\beta-d} d\eta = \frac{\pi^{\frac{d}{2}}\Gamma(\frac{\beta}{2})}{\Gamma(\frac{d}{2})}\,. 
\end{equation*}
Then with the change of variable $\sqrt{s_{i+1}-s_i} \xi_i \rightarrow \eta_i$, we have 
\begin{eqnarray*}
n!\|f_n(\cdot,t,x)\|^2_{\mathcal{H}_1^{\otimes n}} &\leq&  \frac{ (4 \lambda^2 \Gamma_t \Lambda_{\beta})^n}{(2\pi)^{nd}}\int_{T_n(t)} \int_{\RR^{nd}} \prod_{i=1}^n \left(e^{-|\eta_i|^2} |\eta_i|^{\beta-d}\right) \prod_{i=1}^n (s_{i+1}-s_i)^{-\frac{\beta}{2}}d\eta ds\\
&=& \frac{ (4 \lambda^2 \Gamma_t \Lambda_{\beta} \pi^{\frac{d}{2}})^n}{(2\pi)^{nd} \Gamma(\frac{d}{2})^{n}} \Gamma(\frac{\beta}{2})^n \int_{T_n(t)} \prod_{i=1}^n (s_{i+1}-s_i)^{-\frac{\beta}{2}}  ds\\
&=& \frac{ (4 \lambda^2 \Gamma_t \Lambda_{\beta} \pi^{\frac{d}{2}})^n}{(2\pi)^{nd} \Gamma(\frac{d}{2})^n} \Gamma(\frac{\beta}{2})^n \frac{\Gamma(1-\frac{\beta}{2})^n t^{n (1-\frac{\beta}{2})}}{\Gamma((1-\frac{\beta}{2})n+1)}\,.
\end{eqnarray*}
To alleviate the notation we denote 
\begin{equation*}
B:= \frac{\Lambda_{\beta} \Gamma(\frac{\beta}{2}) \Gamma(1-\frac{\beta}{2})}{ 2^{d-2} \pi^{\frac{d}{2}} \Gamma(\frac{d}{2}) }\,.
\end{equation*}
By the log-convexity of Gamma function, we know that $\Gamma(1+x)^2 \leq \Gamma(1+2x)$. Then we obtain 
\begin{eqnarray*}
\|v(t,x)\|_{L^p(\Omega)}& \leq& \sum_{n=0}^{\infty}  (B(p-1) \lambda^2 \Gamma_t )^{\frac{n}{2}} \frac{t^{n   \frac {2-\beta}4}}{\Gamma((1-\frac{\beta}{2})n+1)^{\frac{1}{2}}}\\
&\leq&\sum_{n=0}^{\infty}  (B(p-1) \lambda^2 \Gamma_t )^{\frac{n}{2}} \frac{t^{n (\frac{1}{2}-\frac{\beta}{4})}}{\Gamma( n \frac {2-\beta} 4 +1)}\,.
\end{eqnarray*}
Using the asymptotic behavior of Mittag-Leffler function (see e.g., page 208 in \cite{EMOT})
\begin{equation*}
\sum_{n=0}^{\infty} \frac{z^n}{\Gamma(1+ a n)} = \frac{1}{a}\exp (z^{\frac{1}{a}}) + O (|z|^{-1})  \quad \text{as} \  z\to \infty \,,
\end{equation*}
and considering the fact that we are only interested in when $t \to \infty$, we obtain
\begin{eqnarray*}
\|v(t,x)\|_{L^p(\Omega)} \leq  \frac 4{2-\beta} \exp \left( (B(p-1)\lambda^2 \Gamma_t)^{\frac{1}{1-\frac{\beta}{2}}} t \right)+  O(t^{-\frac {2-\beta}4 })\,.
\end{eqnarray*} 
Thus, as in the proof of Theorem \ref{thm:cor upper}, we obtain
\begin{eqnarray*}
\EE u^p(t,x) \leq \left(   \frac4 {2-\beta} \right)^{\frac{p}{2}} \exp \left( \frac p2 (B(p-1)\lambda^2 \Gamma_t)^{\frac 2 {2-\beta}  } t - \frac{p |x|^2}{4t (\varkappa +1)}\right) e^{\frac{p M^2}{4t\varkappa}} \omega_d^{\frac p2} M^{\frac {dp} 2}\frac{1}{(2\pi t)^{\frac{dp}{4}}}\,.
\end{eqnarray*}
Recall that we have set $\vartheta _t = \Gamma_t^{\frac{1}{2-\beta}}$. Thus, if we want 
\begin{eqnarray*}
\limsup _{t \to \infty} \frac{1}{t\vartheta^2_t} \sup_{|x|\geq \varrho t\vartheta_t} \log \EE  u^p(t,x)  \leq  p \left( B (p-1) \lambda^2\right)^{\frac{1}{1-\frac{\beta}{2}}}  - \frac{p \varrho^2}{2 (\varkappa +1)} < 0\,,
\end{eqnarray*}
by invoking Lemma \ref{lemma:sol nonnegative} and letting $\varkappa \to 0$, we conclude that 
\begin{equation*}
\bar{\upsilon}(p)\leq \sqrt{2} \left( B (p-1) \lambda^2  \right)^{\frac{1}{2-\beta}}\,.
\end{equation*}
Finally, if we plug in the value of $B$ and $\Lambda_{\beta}$, the proposition is proved. 

\end{proof}
\begin{remark}
Proposition \ref{prop: Riesz upper}  still holds if we take the fractional kernel $\Lambda(x)=\prod_{i=1}^d |x_i|^{2H_i-2}$ with $\frac{1}{2}< H_i < 1$ for all $i$ and $\beta : = 2d -2\sum_{i=1}^d H_i$ with $0 < \beta < 2$. The order of $p-1$ and $\lambda$ in the upper bound of $\bar{\upsilon}(p)$ will be exactly the same, although the coefficient may be different. 
\end{remark}

Theorem \ref{thm:cor lower} does not cover the case when the noise is white in space. However, if we approximate the Dirac delta function by $p_{\varepsilon}(x)$, we have the following result. 

\begin{proposition}
Assume $d=1$. Let $u(t,x)$ be the solution to equation \eref{SHE} with  nonnegative initial condition $u_0$ being uniformly bounded away from $0$ in the ball $B_{ M}$ and supported in $B_{rM}$, where $r \geq 1$.  
Assume that $\Lambda(x)$ is the Dirac delta function. Set $\vartheta _t = \Gamma_t$, 
fix $\delta \in (0,1)$ and set  $\eta_t = \Gamma_{ t\delta^2}$. Let $\bar{\upsilon}(p) $ and $\underline{\nu}(p)$ be defined in \eref{nu(p) up} and \eref{39}, respectively.  
Then we have
\begin{equation}
\bar{\upsilon}(p) \leq 2\sqrt{2} (p-1) \lambda^2\,.
\end{equation}
If we further assume  \eref{w2} holds, then 
\begin{equation}
 \underline{\nu}(p) \geq \frac{2\sqrt{2}}{e^2 \pi^{\frac{3}{2}}} (1-\delta)^{\frac{3}{2}}\delta^{\frac{1}{2}} (p-1) \lambda^2  \,.
\end{equation}
\end{proposition}
\begin{proof}
The proof of upper bound follows along the same lines as the proof for proposition \ref{prop: Riesz upper} except now $\mu(d\xi) = d\xi$. For the lower bound, we consider the approximation of the Dirac delta function by the heat kernel $p_{\varepsilon}$, and define 
\begin{equation*}
I_{t,p,\epsilon} = \EE_B\left(\prod_{i=1}^p u_0(x+B_t^i) \exp \left(\lambda^2 \sum_{1 \leq i < j \leq p} \int_0^t \int_0^t \gamma(s-r) p_{\varepsilon}(B_s^i -B_r^j) ds dr  \right) \right)\,.
\end{equation*}
Expanding the exponential and using Fourier analysis, one can show that $\EE u(t,x)^p \geq I_{t,p,\varepsilon}$(see \cite{HHNT, HN}) for any $\varepsilon > 0$. Then the proof follows along the same lines as the proof for Theorem \ref{thm:cor lower} except here, we restrict the expectation on the set 
\begin{equation*}
F := \left \{\omega: \sup_{1 \leq i \leq p} |B_t^i+x| \leq M, \sup_{1\leq i \leq p} \sup_{0 \leq s \leq t\delta} |B_s^{i}| \leq \sqrt{\varepsilon} \right\}\,.
\end{equation*}
We omit the details of the proof. 
\end{proof}
\begin{remark}
If condition \eref{w2} does not hold and the limit is finite $\Gamma_{\infty}$, then we need the following additional condition on $M$:
\begin{equation*}
M \geq \frac{ \sqrt{2}\pi^{\frac{5}{2}} e^2}{ 4 \lambda^2 (p-1) (1-\delta) \Gamma_{\infty}}\,.
\end{equation*}
\end{remark}

\end{document}